\documentclass[10pt]{amsart}
\usepackage[utf8]{inputenc}
\usepackage[T1]{fontenc}   
\usepackage{tcolorbox}
\usepackage{amssymb}
\usepackage{amsmath}
\usepackage{amsthm}
\usepackage{graphicx}
\usepackage{stmaryrd}
\usepackage{mathrsfs}
\usepackage{amscd}
\usepackage{bbm}
\usepackage{enumitem}
\usepackage{fncylab}
\usepackage{xcolor}
\usepackage{url}
\usepackage{leftidx}
\usepackage{calc}
\usepackage[width=\textwidth]{caption}
\usepackage[normalem]{ulem}
\usepackage{dsfont}

\theoremstyle{plain}
\newtheorem{theo}{Theorem}

\newtheorem{coroll}{Corollary}
\newtheorem{lemma}{Lemma}
\newcommand{\T}{\mathbf{T}}
\newcommand{\un}{\mathbf{1}}
\def\pentsup#1{\left\lceil #1 \right\rceil}
\def\pentinf#1{\left\lfloor #1 \right\rfloor}

\renewcommand{\L}{\mathbb{L}}
\renewcommand{\P}{\mathbb{P}}

\newcommand{\N}{\mathbb{N}}
\newcommand{\Z}{\mathbb{Z}}
\newcommand{\E}{\mathbb{E}}
\newcommand{\A}{\mathscr{A}}
\title[CFTP for exponentially ergodic one-dimensional PCA]{Coupling from the past for exponentially ergodic one-dimensional probabilistic cellular automata}
\author{Jean B\'erard}
\address{Institut de Recherche Math\'ematique Avanc\'ee, UMR 7501, Universit\'e de Strasbourg et CNRS, 7 rue Ren\'e Descartes, 67\,000 Strasbourg, France}
\email{jberard@unistra.fr}
\thanks{The author would like to thank Y. Spinka for useful discussions, and in particular for pointing the references \cite{Spi,Spi2}.}
\begin{document}
\begin{abstract}
For every exponentially ergodic one-dimensional probabilistic cellular automaton with positive rates, we construct a locally defined coupling-from-the-past flow whose coalescence time has a finite exponential moment. This construction leads to a finite-size necessary and sufficient condition for exponential ergodicity of one-dimensional cellular automata. As a corollary, we prove that every sufficiently small perturbation of an exponentially ergodic one-dimensional cellular automaton is exponentially ergodic. 
\end{abstract}

\maketitle

\section{Introduction}

\subsection{Definitions and the main result}

Probabilistic Cellular Automata (PCA) form a class of discrete-time Markov processes on spaces of the form $\A^{\L}$, where $\L$ is a lattice (typically, $\L=\Z^d$ for some $d \geq 1$), and $\A$ is a finite set called an alphabet (see e.g. \cite{TooVasStaMitKurPir} for a seminal reference on the subject, and \cite{FerLouNar} for a recent overview). In the present paper, we consider one-dimensional PCAs, that is, $\L=\Z$. An element of $\A^{\Z}$ is a bi-infinite sequence $(v(x))_{x \in \Z}$, where $v(x) \in \A$ for all $x \in \Z$, and we equip the set $\A^{\Z}$ with the product topology and product $\sigma-$algebra. The dynamics of the PCA is specified through a transition kernel $\mathcal{K}$ from $\A^{\{-1,0,1 \}}$ to $\A$, so that for every $\mathbf{v}=(v_{-1},v_0,v_1) \in  \A^{\{-1,0,1 \}}$, $\mathcal{K}(\mathbf{v},\cdot)$ is a probability measure on $\A$. Formally, a probabilistic cellular automaton with kernel $\mathcal{K}$ is a discrete-time Markov process $(X_t)_{t}$ on  $\A^{\Z}$, such that:
 \begin{equation}\label{e:PCA-dyn-1} \mbox{The random variables $\left(X_t(x)\right)_{x \in \Z}$ are independent given $X_{t-1}$},\end{equation}
 \begin{equation}\label{e:PCA-dyn-2} \forall  x \in \Z,  \    X_t(x) \sim \mathcal{K}( \pi_{\llbracket x-1,x+1 \rrbracket}(X_{t-1}) ,\cdot ), \end{equation}
where  $\llbracket x,y \rrbracket$ denotes the discrete interval $\llbracket x,y \rrbracket = \{ z \in \Z ; \ x \leq z \leq y \}$, and where, given two sets $I \subset J$, we denote by $\pi_I$ the canonical projection from $\A^J$ to $\A^I$.

Moreover, we say that our PCA satisfies the {\it positive rates} condition when there exists a $w \in \A$ such that \begin{equation}\label{e:positive-rates}\min_{\mathbf{v} \in \A^{\{-1,0,1 \}}}  \mathcal{K}(\mathbf{v},\{w\}) > 0.\end{equation}

One key question about the long-term dynamics of PCAs is that of ergodicity: we say that a PCA is {\it ergodic} when there exists a (necessarily unique) probability distribution $\mu$ on $\A^{\Z}$ such that, for every initial condition $X_0=\xi \in \A^{\Z}$, one has the convergence  $X_t \xrightarrow[t \to +\infty]{d} \mu$. For an ergodic PCA, an important additional question is that of the convergence speed: we say that a PCA is {\it exponentially ergodic} when there exist positive constants $a,b > 0$ such that, for all $t \geq 0$, all $\xi \in \A^{\Z}$, and all $I \subset \Z$, 
\begin{equation}\label{e:expo-ergod}d_{\mbox{\scriptsize TV}} \left(\mbox{Law}(\pi_I(X_t)), \pi_I(\mu) \right) \leq a  | I | e^{- bt},\end{equation}
where $d_{\mbox{\scriptsize TV}}$ denotes the total variation distance between probability measures (see Section \ref{s:couplings}), and where, for a probability measure $\nu$ on $\A^J$ with $I \subset J$, we denote by $\pi_I(\nu)$ the corresponding image probability measure on $\A^I$.

Among the various methods that may be used to prove ergodicity, we focus on the so-called {\it coupling from the past} (CFTP) approach, which has become a popular tool in the context of Markov-chain based numerical methods (see \cite{ProWil}), but had already been used earlier (not under this specific name) to establish ergodicity for a variety of processes -- see e.g. \cite{TooVasStaMitKurPir} in the context of PCAs, or \cite{Lig} in the context of continuous-time interacting particle systems. 

To formalize this approach, we define a CFTP {\it flow} to be a family of random functions\footnote{We use the space $\mathscr{S}$ of functions $\phi \ : \ \A^{Z} \to \A^{\Z}$ for which there exists an $r \geq 1$ such that the value of $\phi(\xi)$ at site $x$ is a function of those values of $\xi(y)$ for which $|y-x| \leq r$ only. Measurability on $\mathscr{S}$ is then defined by viewing elements of $\mathscr{S}$ as countable collections of functions from $\A^{\llbracket x-r, x+r \rrbracket}$ to $\A$.} $ \left(\Phi_{t_n}^{t_{n+1}} \right)_{n \geq 0}$, where $(t_n)_{n \geq 0}$ is a decreasing integer-valued sequence such that $t_0=0$, and
where $\Phi_{t_n}^{t_{n+1}} \ : \ \A^{\Z} \to \A^{\Z}$ is such that, for all $n \geq 1$, and all $\xi \in \A^{\Z}$, the sequence $\xi, \Phi_{t_{n-1}}^{t_n}(\xi), \ldots,  \Phi_{t_{0}}^{t_1} \circ  \cdots \circ \Phi_{t_{n-1}}^{t_n}(\xi)$,
has the same distribution as $X_{t_n},X_{t_{n-1}},\ldots, X_{t_0}$, starting from $X_{t_n}=\xi$. The {\it coalescence time} of the flow at a site $x \in \Z$ is then defined as $$T_x = \inf \{n \geq 1 ; \ \pi_x \circ \Phi_{t_{0}}^{t_1}  \circ \cdots \circ \Phi^{t_n}_{t_{n-1}} \mbox{ is a constant function} \},$$ with the convention $\inf \emptyset = +\infty$, and where we use the notation $\pi_x$ instead of $\pi_{\{x\}}$ when $I=\{x\}$. If for all $x$, one has that $T_x<+\infty \mbox{ a.s.}$, then the PCA is ergodic. Moreover, if the tail of $T_x$ satisfies an inequality of the form $\P(T_x > t) \leq a e^{-bt}$ for all $t \geq 0$ (with $a$ and $b$ not depending on $x$), one gets the bound \eqref{e:expo-ergod} with the same constants $a$ and $b$.

Our main result is a converse to this property. It states that, whenever a PCA is exponentially ergodic and has positive rates, it is possible to define a CFTP flow for which $T_x$ has a finite exponential moment (uniformly bounded over $x$), and which is, in a precise sense, locally defined. 
\begin{theo}\label{t:CFTP}
Consider an  exponentially ergodic one-dimensional PCA with positive rates. Then there is a CFTP flow with $t_n=-n \cdot L$ for a certain integer $L$, enjoying the following properties:
\begin{itemize}
\item[(i)] for all $x \in \Z$ and $t \geq 0$, $\P(T_x>t) \leq c e^{-dt}$ where $c>0$ and $d>0$ do not depend on $x$;
\item[(ii)] the family of random functions $\left(\Phi_{t_n}^{t_{n+1}}\right)_{n \geq 0}$ is i.i.d.;
\item[(iii)] there exists an i.i.d. family of random variables $(V_y)_{y \in \textstyle{\frac{1}{2}}\Z}$, and a (measurable) function $F$, such that,  for all $x \in \Z$ and $\xi \in \A^{\Z}$, one can write the value of $\Phi_{t_0}^{t_1}(\xi)$ within $\llbracket x(2L)-L,x(2L)+L \rrbracket$ as:
$$\pi_{\llbracket x(2L)-L,x(2L)+L \rrbracket} \left(\Phi_{t_0}^{t_1} (\xi)\right) = F\left(V_{x-1/2},V_x,V_{x+1/2},\pi_{\llbracket (x-1)(2L),(x+1)2L \rrbracket}(\xi) \right).$$
\end{itemize}
\end{theo}
Property (i) merely states the exponential bound on the tail of the coalescence time. Property (ii) is a locality property of the flow with respect to time: the flow is defined on a regular time-grid with mesh $L$, with an i.i.d. structure over distinct time cells. Finally, property (iii) is a locality property with respect to space: over a grid with mesh $2L$, the flow only involves the value of the initial condition and an auxiliary i.i.d. structure within a bounded window.

The conclusion of Theorem \ref{t:CFTP} is already known to hold, under a stronger form, in the case of a {\it monotone} PCA (i.e. when the kernel $\mathcal{K}$ is stochastically monotone with respect to a total order on $\A$ and the corresponding partial product order on $\A^{\{-1,0,1\}}$), as observed in \cite{vdBSte}. In such a case, ergodicity alone is enough to guarantee the existence of a CFTP flow, one can take $t_n=-n$, and $\pi_x((\Phi_{0}^{-1}(\xi))$ can be written as $F(V_x,\pi_{\llbracket x-1,x+1\rrbracket}(\xi))$; moreover, the tail of the coalescence time precisely matches the actual speed of convergence to the limiting distribution.

Still, to our knowledge, a result as general as Theorem \ref{t:CFTP} -- where no other assumption beyond exponential ergodicity and positive rates is needed -- is new, and, except in the monotone case just discussed, only sufficient (but not necessary) conditions for the existence of such a CFTP flow were known (see e.g. \cite{BusMaiMar, MarSabTaa}). Moreover, it is still an open question (see Problems 6.1 and 6.2 in \cite{MarSabTaa}) whether ergodic but not exponentially ergodic PCA exist, so Theorem \ref{t:CFTP} can in fact be applied to every known example of an ergodic one-dimensional PCA. 

\subsection{Consequences}

A direct consequence of Theorem \ref{t:CFTP} is the existence of an algorithm to perfectly sample from the invariant distribution $\mu$ of any exponentially ergodic one-dimensional PCA with positive rates. Also, using the results\footnote{Note that, in \cite{Spi}, the term "exponentially ergodic PCA" is used to refer to the existence of a suitable CFTP flow, whereas in the present paper, exponential ergodicity is a mixing property from which we have to deduce the existence of the CFTP flow.} in \cite{Spi}, we deduce that, if  $X \sim \mu$, the joint distribution of $\left(\pi_{x(2L),(x+1)2L}(X)\right)_{x \in \Z}$ admits a representation as a finite factor of a finite-valued i.i.d. process (it is unclear whether this can be strengthened to prove that $\mu$ itself enjoys this property).

Next, we observe that the flow constructed in the proof of Theorem \ref{t:CFTP} leads to a {\it finite-size} necessary and sufficient condition for exponential ergodicity (with positive rates). Specifically, the proof shows that, assuming positive rates, exponential ergodicity is equivalent to the existence of an integer $L \geq 1$ such that \begin{equation}\label{e:finite-size-cond}\rho=(4L+1)\P \left( \pi_{\llbracket -L,L \rrbracket}  \circ \Phi_{0}^{-L}  \mbox{ is a constant function}\right)<1.\end{equation}  As a consequence, at least in principle, the property of being an exponentially ergodic PCA can always be checked using an algorithm that explores larger and larger values of $L$ (and of the other relevant parameters used in the construction), and stops when a value of $\rho<1$ has been found. Another consequence of this finite-size condition is that exponential ergodicity (with positive rates) is a robust property with respect to sufficiently small perturbations of the dynamics, as stated in the following corollary.

\begin{coroll}\label{c:perturb}
If the kernel $\mathcal{K}$ defines an exponentially ergodic one-dimensional PCA with positive rates, it is also the case of any kernel $\mathcal{K}'$ that is a sufficiently small perturbation of $\mathcal{K}$.
\end{coroll}

\subsection{Discussion}

Theorem \ref{t:CFTP} holds for exponentially ergodic one-dimensional PCAs, and it is indeed a natural question whether an analogous result holds in dimension $d \geq 2$. 

One place\footnote{But not necessarily the only place, see also Lemma \ref{l:debut}.} where the proof of Theorem \ref{t:CFTP} seems to rely heavily on the one-dimensional setting is Lemma \ref{l:bc-coupl}, where we show that exponential ergodicity implies the existence of a coupling with good coalescence properties for the dynamics with boundary conditions. The proof of the lemma uses the fact that the number of sites within a fixed distance of the boundary of a $d-$dimensional box does not grow with the size of the box, which is specific to $d=1$. (This is reminiscent of the proof in \cite{MarOliSch} that "weak mixing implies strong mixing for squares" in the context of two-dimensional spin systems, where here we have one dimension of space and one of time instead of two dimensions of space.) Using stronger mixing conditions (involving the dynamics with boundary conditions) may allow to extend the conclusion of Theorem \ref{t:CFTP} to dimensions $d \geq 2$, but it is unclear how such mixing conditions could be related to more familiar ones in the context of PCAs such as \eqref{e:expo-ergod}. Note that, in the distinct but related context of Markov random fields, the use of "strong" mixing conditions to build CFTP structures and/or perfect simulation algorithms is an active research topic (see e.g. \cite{Spi2, AnaJer}, and the references therein).

Another interesting extension would be to the case of (continuous-time) interacting particle systems, for which the deterministic bound on the speed of propagation of information in PCA dynamics does not hold.

\subsection{Organization of the paper}

The paper is essentially self-contained. Section \ref{s:couplings} contains definitions and simple but useful results on couplings (no claim at originality is made there). Section \ref{s:trapezoid} is devoted to definitions related to PCA dynamics within trapezoids, which are heavily used in the subsequent proofs. Section \ref{s:proof} contains a succession of lemmas leading to the proof of Theorem \ref{t:CFTP} and Corollary \ref{c:perturb}.

\section{Couplings}\label{s:couplings}

Given a finite set $S$ and a finite family $(\nu_e)_{e \in E}$ of probability distributions on $S$, a {\it coupling} of $(\nu_e)_{e \in E}$ is a family of $S-$valued random variables $(Z_e)_{e \in E}$, such that $Z_e \sim \nu_e$ for all $e \in E$. Alternatively, we may view such a coupling as a random map $\Psi \ : \ E \to S$, where $\Psi(e)=Z_e$.

Given two probabilities $\nu_1, \nu_2$ on $S$,  remember the definition of the total variation distance $d_{\mbox{\scriptsize TV}}(\nu_1,\nu_2) = \frac{1}{2} \sum_{s \in S} |\nu_1(s) - \nu_2(s)|$. It is a classical result that the total variation distance is the minimum value of $\P(Z_1=Z_2)$ over all couplings of $\nu_1,\nu_2$. The following two lemmas provide two useful variations over this kind of result.

\begin{lemma}\label{l:couplage-dessous}
Assume that, for a certain $0 < \epsilon < 1$, there exists an $e_0 \in E$ such that one has $d_{\mbox{\normalfont \scriptsize TV}} (\nu_e,\nu_{e_0}) \leq \epsilon/|S|$ for all $e \in E$. Then, for all $ \gamma \in ]\epsilon,1[$, there  
exists a coupling of $(\nu_e)_{e \in E}$ such that $\P(Z_e=Z_{e_0} \mbox{ for all $e \in E$}) \geq (1-\gamma)(1-\epsilon/\gamma)$.
\end{lemma}

\begin{proof}
Let $A = \{ s\in S ; \ \nu_{e_0}(s) \leq \epsilon/(\gamma |S|) \}$. One has that $\nu_{e_0}(A) = \sum_{s \in A} \nu_{e_0}(s) \leq |A | \cdot \epsilon / (\gamma |S|) \leq \epsilon/\gamma$. Now, if $s \in A^c$, one has that  $ \nu_{e_0}(s) \geq \epsilon/(\gamma |S|)$, so that, since
$|\nu_{e}(s) -  \nu_{e_0}(s)| \leq  d_{\mbox{\scriptsize TV}} (\nu_{e},\nu_{e_0}) \leq \epsilon/|S|$, one has $\nu_{e}(s) \geq  \nu_{e_0}(s) - \epsilon/|S| \geq (1-\gamma) \nu_{e_0}(s)$. Now consider a pairwise disjoint family $(\mathfrak{I}(s))_{s \in A^c}$ of subintervals of $[0,1]$, with respective lengths $(1-\gamma) \nu_{e_0}(s)$, Then, for every $e \in E$, complete these intervals into a partition of $[0,1]$ by adding pairwise disjoint intervals $(\mathfrak{L}_e(s))_{s \in A^c}$, with respective lengths $\nu_{e}(s)-(1-\gamma) \nu_{e_0}(s)$, and pairwise disjoint intervals $(\mathfrak{K}_e(s))_{s \in A}$, with respective lengths  $\nu_{e}(s)$. Now consider a random variable $U$ with uniform distribution on $[0,1]$.
Whenever $U$ belongs to the interval $\mathfrak{I}(s)$, we set $Z_e=s$ for all $e \in E$. When $U$ does not belong to $\bigcup_{s \in A^c} \mathfrak{I}(s)$, for a given $e$, either $U$ belongs to a (unique) interval $\mathfrak{L}_e(s)$, or to a (unique) interval $\mathfrak{K}_e(s)$, and we define $Z_e$ as precisely the corresponding $s$. It is now apparent that each $Z_e$ has $\nu_{e}$ as its distribution, while 
$ \P(Z_{e_0}=Z_e  \mbox{ for all $e \in E$}) \geq \P(U \in \bigcup_{s \in A^c} \mathfrak{I}(s)) = \sum_{s \in A^c} (1-\gamma) \nu_{e_0}(s) = (1-\gamma) \nu_{e_0}(A^c) \geq (1-\gamma)(1-\epsilon/\gamma)$.

\end{proof}

\begin{lemma}\label{l:couplage-dessus}
Assume that, for a certain $0 < \epsilon < 1$, there exists an $e_0 \in E$ such that one has $d_{\mbox{\normalfont \scriptsize TV}} (\nu_e,\nu_{e_0}) \leq \epsilon$ for all $e \in E$. Then there exists a coupling of $(\nu_e)_{e \in E}$ such that, for all $J \subset E$,  $\P(Z_e=Z_{e_0} \mbox{ for all $e \in J$}) \geq 1-|J| \epsilon$.
\end{lemma} 

\begin{proof}
We recycle the classical coupling construction leading to the probability of equality between a pair of random variables being equal to the total variation distance. First consider a partition of the interval $[0,1]$ into a pairwise disjoint family  $(\mathfrak{J}_{e_0}(s))_{s \in S}$ of subintervals, with respective lengths $\nu_{e_0}(s)$. For $e \in E \setminus \{e_0 \}$, let $A_e =  \{ s\in S ; \ \nu_{e_0}(s) \geq \nu_e(s)) \}$. For $s \in A_e$, let $\mathfrak{J}_e(s)$ be a  subinterval of $[0,1]$ with length $\nu_e(s)$ such that $\mathfrak{J}_e(s) \subset \mathfrak{J}_{e_0}(s)$. Then, for $s \in A_e^c$, let $\mathfrak{J}_e(s)$ be the union of a finite number of disjoint subintervals of $[0,1]$, in such a way that $\mathfrak{J}_{e_0}(s) \subset \mathfrak{J}_{e}(s)$, that the total length of $\mathfrak{J}_{e}(s)$ equals $\nu_e(s)$, and that the family  $(\mathfrak{J}_e(s))_{s \in S}$ forms a partition of $[0,1]$. Using a random variable $U$ with uniform distribution on $[0,1]$, and defining $Z_e$ as the unique $s$ such that $\mathfrak{J}_e(s)$ contains $U$, we have that $Z_e \sim \nu_e$ for all $e \in E$, and, for all $e \in  E \setminus \{e_0 \}$, $\P(Z_{e_0}=Z_e) = d_{\mbox{\scriptsize TV}} (\nu_e,\nu_{e_0})$.
Thus, $\P(\exists \  e \in J \mbox{ such that }Z_{e_0} \neq Z_e) \leq \sum_{e \in J} \P(Z_{e_0} \neq Z_e) =  \sum_{e \in J} d_{\mbox{\scriptsize TV}} (\nu_e,\nu_{e_0}) \leq |J| \epsilon$.
\end{proof}

In the sequel, a coupling provided by Lemma \ref{l:couplage-dessous} (resp. Lemma \ref{l:couplage-dessus}) will be called a type I (resp. type II) coupling.

\section{Trapezoids}\label{s:trapezoid}

In this paper, we use the generic term {\it trapezoid} to refer to discrete isoceles trapezoids drawn on the space-time lattice $\Z \times (-\N)$ whose lateral sides have their respective slopes equal either to $-1,+1$ or to $+1,-1$ , as shown in Fig. \ref{f:up-down}. We distinguish between {\it downward} trapezoids (when the top is longer than the base), and {\it upward} trapezoids (when the top is shorter than the base), with time flowing from top to bottom.

We define the {\it outer boundary} of an upward trapezoid $\T$ as the union, on both sides, of the two discrete segments parallel to the lateral sides of $\T$, at horizontal distance respectively $1$ and $2$ from $\T$, starting at the ordinate of the top, and stopping one unit above the ordinate of the base. The outer boundary is denoted by $\partial_{+} \T$. We also use the notation $\T(m)=\T \cap (\Z \times \{m\})$.

\begin{figure}
\centering
\includegraphics[width=10cm]{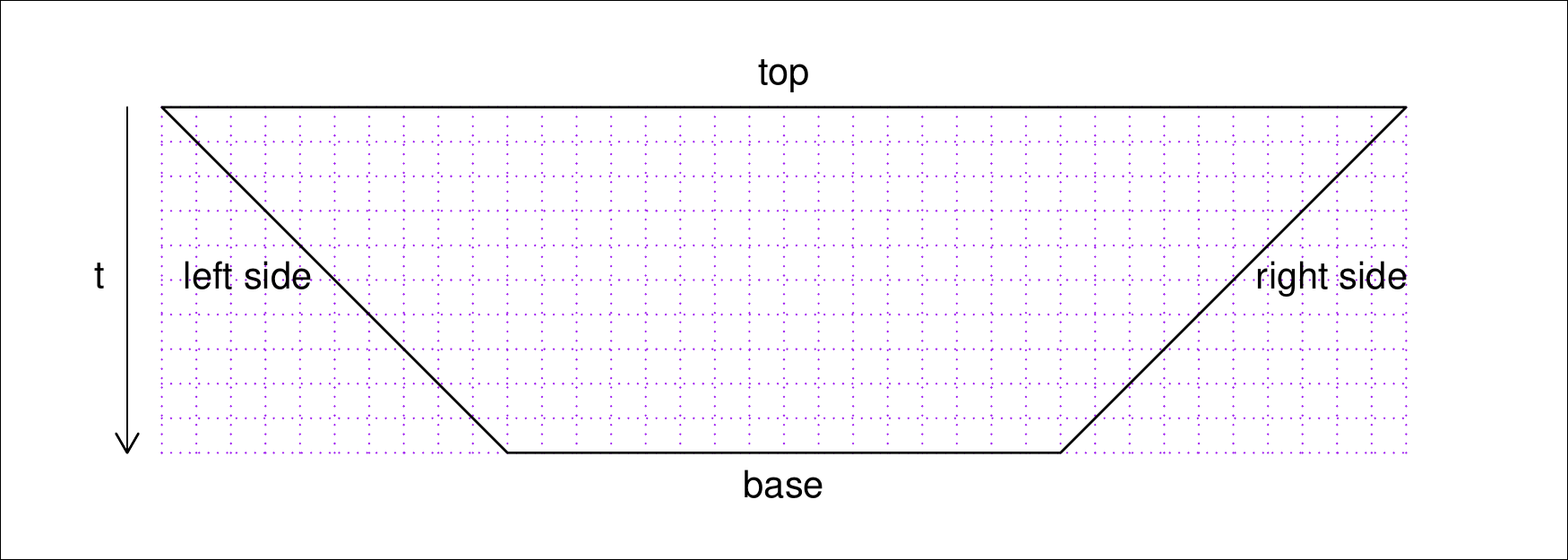}

\vspace{1em}

 \includegraphics[width=10cm]{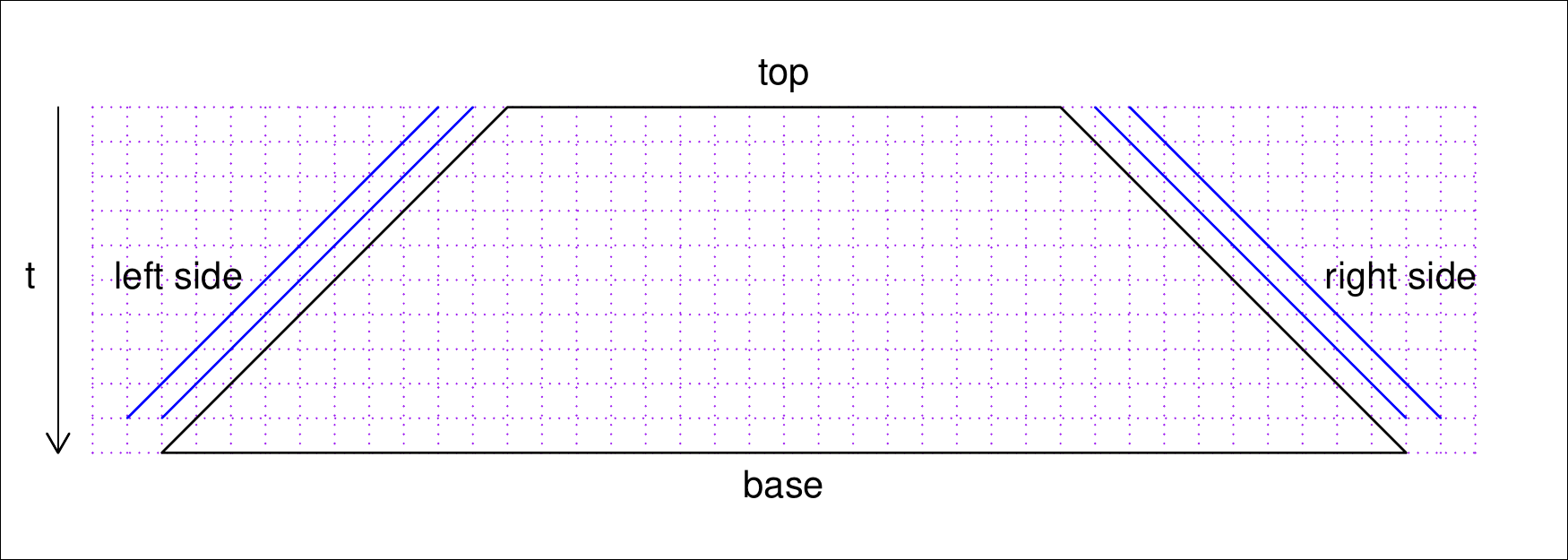}
\caption{A downward trapezoid (above) and an upward trapezoid (below). In the upward case,  the outer boundary is shown in blue. The integer lattice is drawn using purple dotted lines.}
\label{f:up-down}
\end{figure}

\subsection{Dynamics within a trapezoid}

\subsubsection{Downward case}

Consider a downward trapezoid $\T$  with height $L$ and base-length $K$, with  $\mbox{top}(\T)=\llbracket z-L, z+K+L \rrbracket \times \{ \tau \} = \T(\tau) $, and $\mbox{base}(\T)=  \llbracket z, z+K \rrbracket \times \{ \tau +L \} = \T(\tau+L)$. Starting from a configuration $\zeta$ consisting of an element of $\A$ at each site of the top of $\T$, we define the PCA dynamics within $\T$ as a Markov process on the successive state spaces $\A^{\T(\tau)},\ldots, \A^{\T(\tau+L)}$ in which, given the configurations within $ \T(\tau+m)$, where $0 \leq m \leq L-1$, the configuration within $ \T(\tau+m+1)$ is obtained by following \eqref{e:PCA-dyn-1}-\eqref{e:PCA-dyn-2}, for $x \in \llbracket z-L+m+1, z+K+L-m-1 \rrbracket$. We denote by $G(\zeta,\T)$ the resulting overall distribution on $\A^{\T}$.

The following restriction property shows that the dynamics within $\T$ we have just defined, coincides with the restriction of the overall dynamics of the PCA within $\T$, conditional upon a suitably defined "outside" of $\T$. The proof is omitted, and is an easy consequence of e.g. the basic coupling described in Subsection \ref{ss:basic} below.

\begin{lemma}\label{l:restriction-downward}
Consider $s \leq \tau$, and define $ \mbox{out}_s(\T)$ as the set of $(x,t)$ such that either $s \leq t < \tau$, or $\tau \leq t \leq \tau+L$ and the horizontal distance from $(x,t)$ to the boundary of $\T$ is $\geq  t-\tau$.
Starting from $X_s=\xi$ at a time $s \leq \tau$, the distribution of $(X_t(x))_{(x,t) \in \T}$, conditional upon $(X_t(x))_{(x,t) \in \mbox{out}_s(\T)}$, is  $G(\zeta,\T)$, with $\zeta=(X_t(x), (x,t) \in \mbox{base}(\T))$.
\end{lemma}

\subsubsection{Upward case}

Consider an upward trapezoid $\T$ with height $L$ and top-length $M$, $\mbox{top}(\T)= \llbracket z, z+M \rrbracket  \times \{ \tau \}= \T(\tau)$  and $\mbox{base}(\T)=  \llbracket z-L, z+M+L \rrbracket \times \{ \tau +M \}  = \T(\tau+L)$. Starting from a configuration $\zeta$ consisting of an element of $\A$ at each site of the top of $\T$, and a boundary condition $\chi$ consisting of an element of $\A$ at each site of the outer boundary $\partial_{+} \T$, we can define the PCA dynamics within $\T$ as in the previous case: given the configurations within $ \T(\tau+m)$, where $0 \leq m \leq L-1$, the configuration  within $ \T(\tau+m+1)$ is obtained by following \eqref{e:PCA-dyn-1}-\eqref{e:PCA-dyn-2}, for $x \in \llbracket z-m-1, z+M+m+1 \rrbracket$, using the boundary condition to make sense of \eqref{e:PCA-dyn-2} when $x \in \{z-m-1,z-m,z+M+m,z+M+m+1\}$. We denote by $G_{\chi}(\zeta,\T)$ the resulting overall distribution on $\A^{\T}$. 

We now state a restriction property for the dynamics with boundary conditions on $\T$. The proof is similar to that of Lemma \ref{l:restriction-upward}.

\begin{lemma}\label{l:restriction-upward}
Consider $s \leq \tau$, and define $ \mbox{out}_s(\T)$ as the set of $(x,t) \in \llbracket s, \tau+L \rrbracket \times \Z $ such that either $(x,t) \notin \T$ or $(x,t) \in \mbox{base}(\T)$. 
Starting from $X_s=\xi$ at a time $s \leq \tau$, the distribution of $(X_t(x))_{(x,t) \in \T}$, conditional upon $(X_t(x))_{(x,t) \in \mbox{out}_s(\T)}$, is  $G_{\chi}(\zeta,\T)$, with $\zeta=(X_t(x), (x,t) \in \mbox{base}(\T))$ and 
$\chi=(X_t(x), (x,t) \in \partial_+(\T)$.
\end{lemma}

\subsection{Coupling within a trapezoid}

\subsubsection{Downward $(K,L)-$coupling}

Given a downward trapezoid $\T$  with height $L$ and base-length $K$, we define a downward $(K,L)-$coupling to be a coupling of the dynamics within $\T$, for every possible initial configuration on the top, that is, a coupling of the family $G(\zeta,\T), \ \zeta \in \A^{\mbox{top}(\T)}$. Note that, given a coupling for the configuration at the base, i.e. a coupling for the family $\pi_{\mbox{base}(\T)} (G(\zeta,\T)), \ \zeta \in \A^{\mbox{top}(\T)}$,
one can always build a full $(K,L)-$coupling by sampling from the distribution of the whole dynamics within $\T$ starting from $\zeta$, conditional upon the random configuration at the base generated by the coupling. If $\Psi$ denotes a random function from $\A^{\mbox{top}(\T)}$ to $\A^{\T}$ corresponding to a $(K,L)-$coupling, we say that {\it coalescence} occurs when $ \pi_{\mbox{base}(\T)} \circ  \Psi$ is a constant function, and we say that an $(x,t) \in \mbox{base}(\T)$ is {\it locked} when $\pi_{(x,t)} \circ \Psi$ is a constant function.

\subsubsection{Upward $(M,L)-$coupling}

For an upward trapezoid $\T$ with height $L$ and top-length $M$, we define an upward $(M,L)-$coupling to be a coupling of the dynamics within $\T$ for every possible boundary condition, and every possible initial configuration on the top of $\T$, i.e. a coupling for the family $G_{\chi}(\zeta,\T), \ \zeta \in \A^{\mbox{top}(\T)}, \chi \in \A^{\partial_{+} \T}$. 
As above, a coupling for the configuration at the base is enough to define a full $(M,L)-$coupling. If $\Psi$ denotes a random function corresponding to an $(M,L)-$coupling, we say that {\it coalescence} occurs for the boundary condition $\chi$ when $\zeta \mapsto \pi_{\mbox{base}(\T)} (\Phi(\zeta,\chi))$ is a constant function, and we say that $(x,t) \in \mbox{base}(\T)$ is {\it locked} for the boundary condition $\chi$ when $\zeta \mapsto \pi_{(x,t)} (\Phi(\zeta,\chi))$ is a constant function.

\subsubsection{The basic coupling}\label{ss:basic}

The basic coupling provides a simple way of defining couplings for the PCA dynamics.
It is defined through an i.i.d. family of random functions $(\Gamma_{x,t})_{x \in \Z, t \in \Z}$, where  $\Gamma_{x,t}   \   :    \      \A^{\{-1,0,1 \}} \to \A$ is such that, for all $\mathbf{v} \in  \A^{\{-1,0,1 \}}$, the law of $\Gamma_{x,t}(\mathbf{v})$ is $\mathcal{K}(\mathbf{v},\cdot)$. Moreover, thanks to the positive rates property \eqref{e:positive-rates}, we may assume that there is a $\kappa>0$ and a $w \in \A$ such that \begin{equation}\label{e:positive-rates-coupl}\P( \Gamma_{x,t}(\mathbf{v})=w \mbox{ for all } \mathbf{v} ) \geq \kappa.\end{equation}

(It is easy to explicitly design such functions, using a single random variable $U_{x,t}$ with uniform distribution on $[0,1]$ and a suitable partition of $[0,1]$ into sub-intervals for each $\mathbf{v}$). Conditions \eqref{e:PCA-dyn-1}-\eqref{e:PCA-dyn-2} are then implemented through the equation:
$$X_t(x) = \Gamma_{x,t} \left( \pi_{\llbracket x-1,x+1 \rrbracket}(X_{t-1}) \right).$$

Using the basic coupling, we can easily design downward $(K,L)-$ or upward $(M,L)-$couplings, but these may not enjoy the coalescence properties we are after. We shall nevertheless use the basic coupling on parts of the trapezoids we consider, using the restriction properties contained in Lemmas \ref{l:restriction-downward} and \ref{l:restriction-upward} to patch together couplings defined on different parts.

\section{Proof of the main results}\label{s:proof}

Our first lemma shows that, for downward trapezoids with a sufficiently large height-to-base ratio, one has a coupling with suitable control over the non-coalescence probability.

\begin{lemma}\label{l:debut}
There exist constants $\alpha > 0$, $c_1>0$, $d_1>0$ such that, for all large enough $K$, and all $L \geq \alpha K$, one can define a downward $(K,L)-$coupling such that the probability of non-coalescence is bounded above by $c_1 \cdot e^{-d_1 K}$.
\end{lemma}

\begin{proof}
For $K \geq 1$, and arbitrary $z$, $\tau$, denote by $\T$ the downward trapezoid with $\mbox{top}=\llbracket z-L, z+K+L \rrbracket \times \{ \tau \}  $, and $\mbox{base}= \llbracket z, z+K \rrbracket \times \{ \tau +L \}  $. We shall apply Lemma \ref{l:couplage-dessous} with $S=\A^{\mbox{base}}$, $E=\A^{\mbox{top}}$, $\mu_e=\pi_{\mbox{base}}(G(e,\T))$, and $e_0$ an arbitrarily chosen element of $E$.

 One has $|S|=|\A|^{K+1}$, and $d_{\mbox{\scriptsize TV}} (\mu_e,\mu_{e_0}) \leq 2a (K+1) \cdot e^{-b L} \leq 2a (K+1) \cdot e^{-b \alpha K}$ for all $e$.
As soon as $\alpha > \log(|\A|)/b$, we see that $\epsilon=2a(K+1) \cdot e^{-b \alpha K} \cdot |\A|^{K+1}$ decays exponentially fast with $K$, so we can apply Lemma \ref{l:couplage-dessous}, with the value of $\epsilon$ just defined, and e.g. $\gamma = \epsilon^{1/2}$, to get the desired coupling.
\end{proof}

It turns out that, to prove Theorem \ref{t:CFTP}, we need to extend Lemma \ref{l:debut} to allow for "flatter" trapezoids, at the price of a slightly worse bound on the coalescence probability. This is done in the following two lemmas, using as a key tool a family of nested self-similar trapezoids with a type I coupling (Lemma \ref{l:self-sim}), followed by a type II coupling (Lemma \ref{l:aplati}).

\begin{lemma}\label{l:self-sim}
For all $A>0$, and for arbitrarily large $L$, there exists a downward $(K,L)-$coupling with the following properties as $L \to +\infty$:
\begin{itemize}
\item[\textbullet]  $L/K \sim (\log L)^{-A}$
\item[\textbullet] the probability that the number of unlocked sites exceeds $K \cdot (\log L)^{-A}$ is bounded above by 
$e^{-  L^{1+o(1)}}$
\end{itemize}
\end{lemma}

\begin{proof}
Let $\alpha > 0$, $c_1>0$ and $d_1>0$ be as in the statement of Lemma \ref{l:debut}.  Let $\ell_0$ be a  even integer number such that $\ell_0 \geq 4 \alpha$, and define inductively the sequences $(\ell_n)_{n \geq 0}$ and $(k_n)_{n \geq 0}$ by 
$k_n = \pentinf{\ell_n/\alpha}_2$, and $\ell_{n+1}=k_n/2+2 \ell_n$, where $\pentinf{m}_2$ stands for the largest even integer number less than or equal to $m$.

These definitions allow one to exactly fit a downward trapezoid with base length $k_{n}$ and height $\ell_n$ into a discrete isoceles triangle with base length $2 \ell_{n+1}$ and height $\ell_{n+1}$, as shown in Fig. \ref{f:trap-tri}.

\begin{figure}
\centering
\includegraphics[width=10cm]{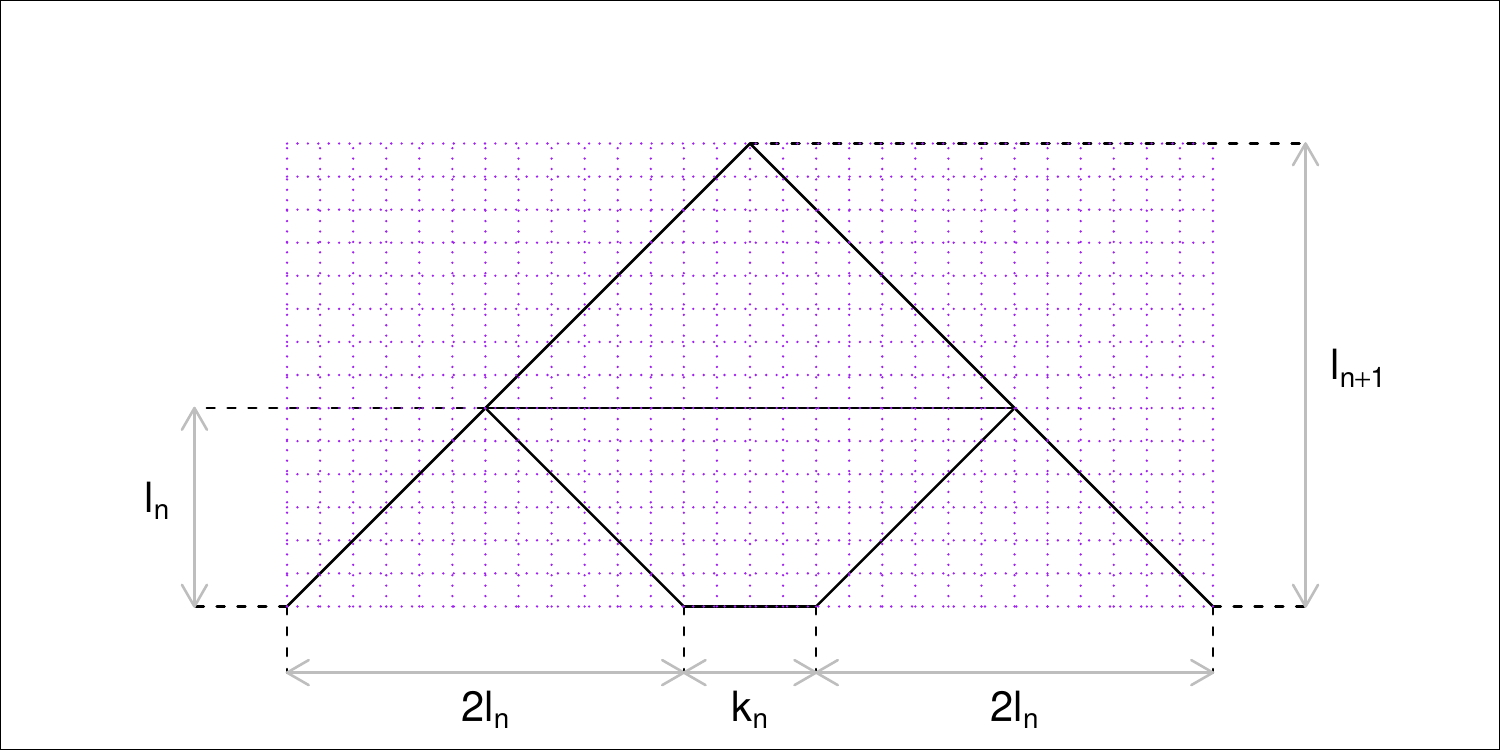}
\caption{Fitting a trapezoid with base length $k_{n}$ and height $\ell_n$ into a triangle with base length $2 \ell_{n+1}$ and height $\ell_{n+1}$. The integer lattice is drawn using purple dotted lines.}
\label{f:trap-tri}
\end{figure}

By definition, we have that, for all $n$, $\ell_{n+1} \geq \ell_n/(2\alpha) -1 + 2 \ell_n = (2+1/(2\alpha)) \ell_n -1$, and we deduce that the sequence $(\ell_n)_{n \geq 0}$ is increasing, 
and that, for all $n \geq 0$, $1/(2 \alpha) \leq k_n / \ell_n \leq 1/\alpha$. 
On the other hand, we have that $(1+1/(2\alpha))^n \ell_0 \leq \ell_n \leq (2+1/(2\alpha))^n \ell_0$.

Now put side-by-side $q$ downward trapezoids with base length $k_n$ and height $\ell_n$. These trapezoids form generation $0$, and fit into a larger downward trapezoid $\T$ of height $L=\ell_n$ and base length $K=q k_n + (2q-2) \ell_n$.
Between two consecutive trapezoids of generation $0$ lies a triangle with base length $2 \ell_n$ and height $\ell_n$. Within every such triangle, we fit a trapezoid with base length $k_{n-1}$ and height $\ell_{n-1}$. These trapezoids form generation $1$. We then iterate the following procedure for $i=1,\ldots, n-1$. Between two consecutive trapezoids of generation $\leq i$ (two consecutive trapezoids may belong to distinct generations) lies a triangle with base length $2 \ell_{n-i}$ and height $\ell_{n-i}$. Within every such triangle, we fit a trapezoid with base length $k_{n-i-1}$ and height $\ell_{n-i-1}$. An illustration is provided in Fig. \ref{f:nest-trap}.

The $0-$th generation trapezoids cover a base of total length $q k_n$, and the triangles between them cover a base of total length $(2q-2) \ell_n$. For $i=0,\ldots, n-1$, going from generation $i$ to generation $i+1$ results in the addition of a new generation of trapezoids with heights $\ell_{n-i-1}$ and base lengths $k_{n-i-1}$, which multiplies the base length previously covered by triangles in generation $i$ by a factor $1-\frac{k_{n-i-1}}{k_{n-i-1} + 4 \ell_{n-i-1}} \leq \frac{8 \alpha}{1+8 \alpha}<1$.

As a result, the total length in the base of $\T$ that is not covered by the base of a trapezoid of whichever generation, is less than $f=(2q-2) \ell_n \left( \frac{8 \alpha}{1+8 \alpha}\right)^n$.

There are $q$ trapezoids in generation number $0$, and $q-1$ in generation number $1$. After generation $1$, each further generation leads to twice as many trapezoids as in the previous one, so the total number of trapezoids is $r=q + (q-1) \cdot (1+2+\cdots + 2^{n-2}) \leq q \cdot 2^n$.

We now define a downward $(K,L)-$coupling inside $\T$, for all large enough $\ell_0$. We use within each trapezoid belonging to generation $n-j$ (with height $\ell_j$ and base length $k_j$), the $(k_j,\ell_j)-$coupling from Lemma \ref{l:debut}, independently from other trapezoids (the fact that $\ell_0$ is large enough, and that, by construction,   $\ell_j \geq \alpha k_j$, ensures that the lemma can be applied for all $j=0,\ldots,n$). 
In the part of $\T$ not belonging to any of the previous trapezoids, we just use the basic coupling. That this is a licit construction leading to a downward $(K,L)-$coupling is a consequence of Lemmas \ref{l:restriction-downward} and \ref{l:restriction-upward}.

For a trapezoid of height $\ell_j$ and base length $k_j$, the probability of non-coalescence of the $(k_j,\ell_j)-$coupling is, according to Lemma \ref{l:debut}, bounded above by $c_1 e^{-d_1 k_j}$. By the union bound, the probability that coalescence does not occur in at least one of the trapezoids, is less than $r c_1 e^{-d_1 k_0}$, and so less than $q \cdot 2^n e^{-d_1 k_0}$. When coalescence occurs in every trapezoid, every unlocked site of our overall $(K,L)-$coupling must belong to the complement of the bases of these trapezoids, whose total length does not exceed $f$. 

Now let $B$ be such that $B \cdot  \log \left(1+\frac{1}{8 \alpha} \right)>A$, let $q=\pentsup{\frac{1}{2+1/\alpha} \cdot (\log \ell_0)^A}$ and let $n=\pentsup{B \cdot \log \log \ell_0}$ (assuming that $\ell_0$ is large enough so that $q \geq 2$ and $n \geq 1$).

From the bound $(2+1/\alpha)^n \ell_0 \leq \ell_n \leq (2+2/\alpha)^n \ell_0$, we see that, as $\ell_0 \to +\infty$, $\log \ell_n \sim \log \ell_0$, and also $q \sim \frac{1}{2+1/\alpha} \cdot (\log \ell_n)^A$.
Remembering that $L=\ell_n$ and that $K=q k_n + (2q-2) \ell_n$, we see that 
$K \sim (\log L)^A L$. Moreover, $f=(2q-2) \ell_n \left( \frac{8 \alpha}{1+8 \alpha}\right)^n \sim \frac{2}{2+1/\alpha} K (\log L)^{-B \log \left(1+\frac{1}{8 \alpha} \right) }  = o\left(K \cdot (\log L)^{-A} \right)$.

Now remember that the probability of having more than $f$ unlocked sites is bounded above by $q \cdot 2^n e^{-d_1 k_0}$. We have $q \sim \frac{1}{2+1/\alpha}  (\log L)^A$, $2^n \sim (\log L)^{B \log 2}$,  and, writing $\ell_0 = e^{\log \ell_0}$, and using the fact that $\log L \sim \log \ell_0$, 
and $k_0 = \ell_0/\alpha + O(1)$, we may write $e^{-d_1 k_0}$ as $e^{- L^{1+o(1)}}$, and absorb both smaller order factors $(\log L)^A$ and $(\log L)^{B \log 2}$ into this expression, so that 
the  probability of having more than $f$ unlocked sites is bounded above by $e^{- L^{1+o(1)}}$.

\begin{figure}
\centering
\includegraphics[width=12cm]{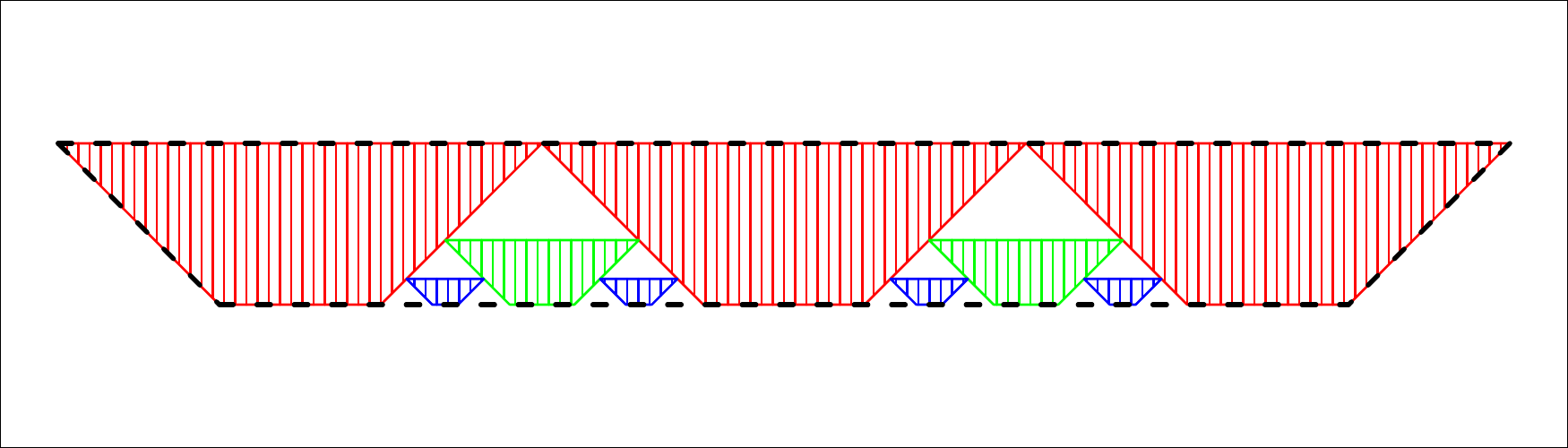}
\caption{Illustration with $q=3$ and $\alpha=1$. The trapezoid $\T$ is drawn with dashed lines. Three generations of nested trapezoids are drawn: generation $0$ (red), generation $1$ (green), generation $2$ (blue).}
\label{f:nest-trap}
\end{figure}

\end{proof}

\begin{lemma}\label{l:aplati}
For any $A > 0$, and for arbitrarily large $L$, there exists a downward $(K,L)-$coupling with the following properties as $L \to +\infty$:
\begin{itemize}
\item[\textbullet]  $L/K \sim h (\log L)^{-A}$ for some constant $h > 0$
\item[\textbullet] the probability that coalescence does not occur is bounded above by 
$e^{- L^{1+o(1)}}$.
\end{itemize}
\end{lemma}

\begin{proof}
Apply Lemma  \ref{l:self-sim} to find a  $(K_0,L_0)-$coupling for an arbitarily large $L_0$, with $L_0/K_0 \sim (\log L_0)^{-A}$ as $L_0 \to +\infty$, and let $t=\pentsup{(h-1) K_0 \cdot (\log L_0)^{-A}}$ for a certain constant $h>1+(\log |\A|)/b$.
We then let the resulting $K_0$ sites at the base evolve according to a type II $(K_0-2t,t)$-coupling (see Lemma \ref{l:couplage-dessus}), independent from the previous $(K_0,L_0)-$coupling.

Conditional upon the $(K_0,L_0)-$coupling, when the number of locked sites is less than $K_0 \cdot (\log L_0)^{-A}$, there are at most $|\A|^{K_0 \cdot (\log L_0)^{-A}}$ distinct initial configurations fed into the top of the type II $(K_0-2t,t)-$coupling. 
In such a case, by Lemma \ref{l:couplage-dessus}, the (conditional) probability that coalescence does not occur within the $(K_0-2t,t)-$coupling is bounded above by $|\A|^{K_0 \cdot (\log L_0)^{-A}} (K_0-2t)2 a e^{-bt}$, which rewrites as  $e^{-L_0^{1+o(1)}}$ since $h>(\log |\A|)/b$ and $L_0/K_0 \sim (\log L_0)^{-A}$
On the other hand, by Lemma \ref{l:self-sim}, the probability that the number of locked sites exceeds $K_0 \cdot (\log L_0)^{-A}$ in the $(K_0,L_0)-$coupling is also bounded above by $e^{-L_0^{1+o(1)}}$.

We have thus built a downward $(K,L)-$coupling with $K=K_0-2t$ and $L=L_0+t$, with $K \sim K_0$ and $L \sim h L_0$, and so $L/K \sim h (\log L_0)^{-A} \sim  h (\log L)^{-A}$.
Moreover, the non-coalescence probability is bounded above by $e^{-L_0^{1+o(1)}}$, and so by  $e^{-L^{1+o(1)}}$.
\end{proof}

We now consider couplings for the dynamics involving boundary conditions within an upward trapezoid.

\begin{lemma}\label{l:bc-coupl}
There exists a constant $\theta > 0$ such that, for all large enough $L$ and $M \leq  L/(\log L)^{\theta}$, there is an
 $(M,L)-$coupling whose non-coalescence probability is bounded above, for any boundary condition, by $e^{-L^{1+o(1)}}$ as $L \to +\infty$, where the $o(1)$ is uniform over $M$ and over the boundary condition.
\end{lemma}

\begin{proof}
 We define a coupling of the dynamics within an upward trapezoid $\T$ with height $L$ and top length $M$, for a given boundary condition $\chi$, assuming that $M \leq L$. Let $t=\pentsup{2(\log L)/b}$, and, for $\pentsup{(\log L)^3} \leq i \leq q$, where $q=\pentinf{L/t}$, consider the slice of $\T$ formed by the upward trapezoid $\T_i$ with height $t$ and base length $M+2it$ (see Fig. \ref{f:pile-trap}).

\begin{figure}
\centering
\includegraphics[width=8cm]{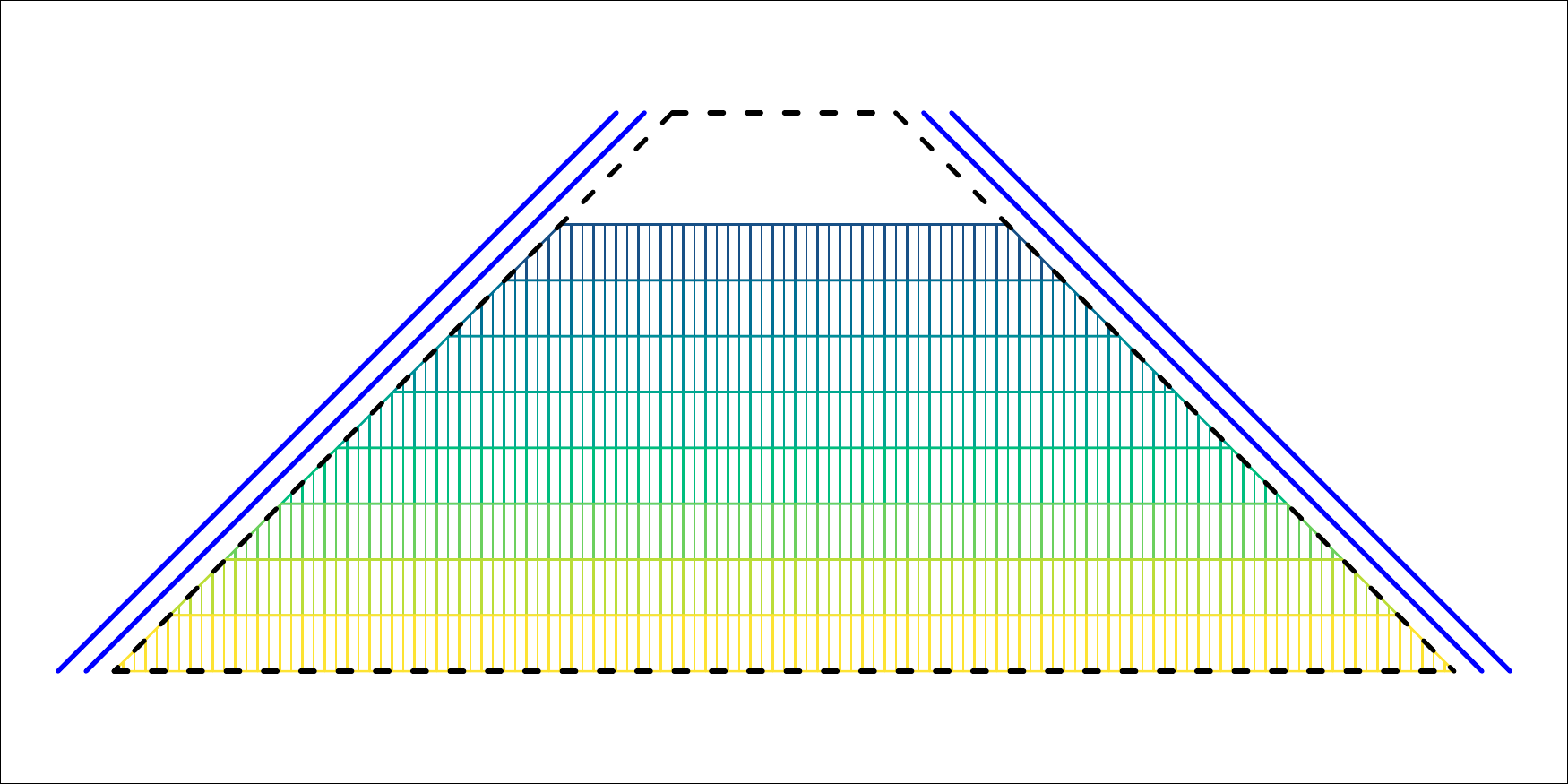}
\caption{Slicing of $\T$ by consecutive trapezoids $\T_i$, depicted in various colours. Boundary conditions are shown in blue.}
\label{f:pile-trap}
\end{figure}

\begin{figure}
\centering
\includegraphics[width=12cm]{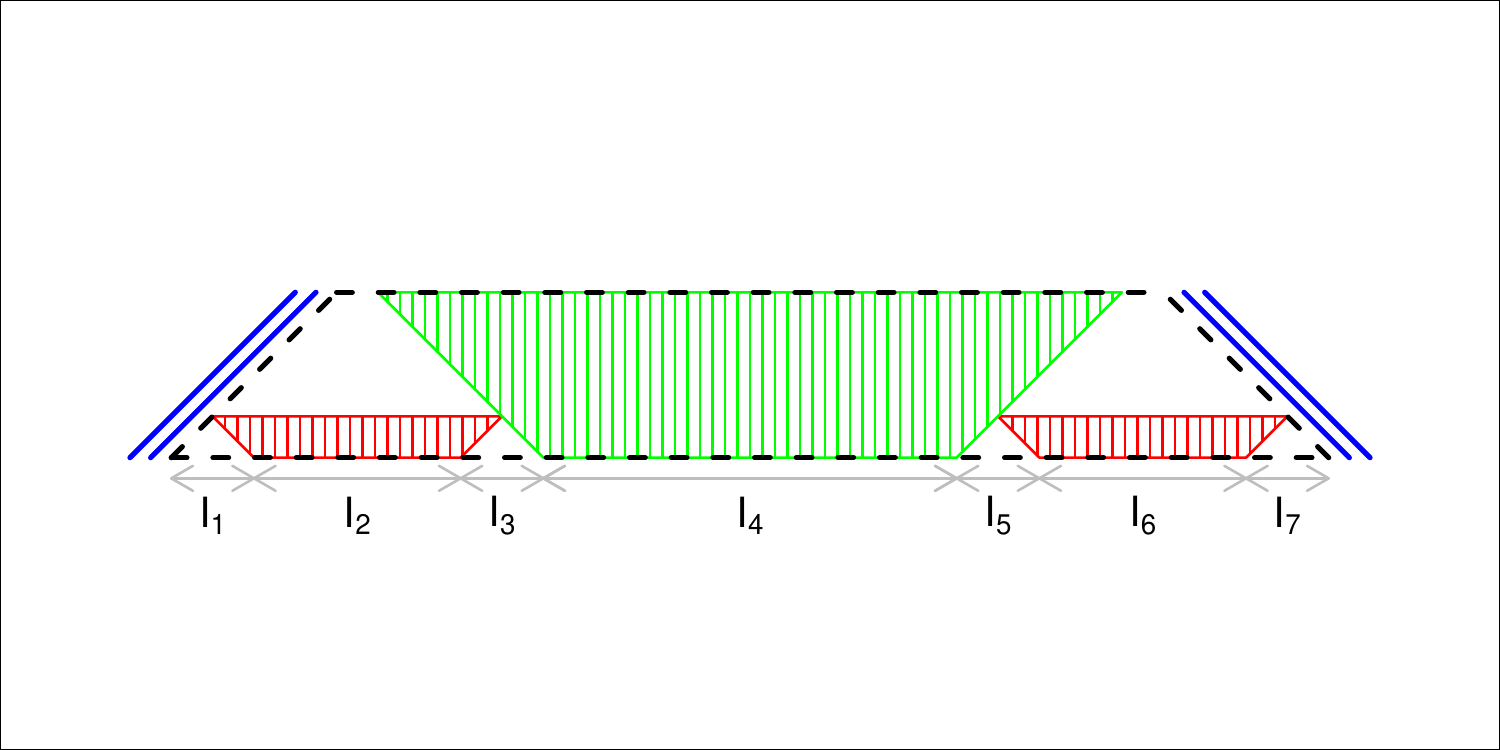}
\caption{Within a trapezoid $\T_i$, intervals $I_1,\ldots, I_7$, and trapezoids $\mathbf{T_{i,1}}$ (red), $\mathbf{T_{i,2}}$ (green), $\mathbf{T_{i,3}}$ (red). Boundary conditions are shown in blue.}
\label{f:empilement}
\end{figure}

For each $i$, we divide the base of $\T_i$ into seven consecutive intervals $I_{1},\ldots, I_{7}$  (from left to right), whose lengths $g_j$ are defined as follows: $g_{1}=g_{3}=g_{5}=g_{7}=k$, with $k=\pentsup{(6/b) \log \log L}_2$ (here $\pentsup{m}_2$ stands for the largest even integer number less than or equal to $m$), $g_{2}=g_{6}=\pentsup{(\log L)^2}$, $g_{4}=2it - (g_1+g_2+g_3+g_5+g_6+g_7)$. (Condition $i \geq (\log L)^3$ ensures that, for all large enough $L$, we have $g_4 \geq 0$.)

We then put two downward trapezoids $\T_{i,1}$ and $\T_{i,3}$ of height $k/2$ on top of $I_2$ and $I_6$ respectively, and a downward trapezoid $\T_{i,2}$ of height $t$ on top of $I_4$. Observe that these trapezoids do not intersect each other except on their boundaries, and, that for all large enough $L$, they do not touch the outer boundary of $\T$ (see Fig. \ref{f:empilement}).

We now define by induction the coupling within $\T$. To begin with, above $\T_{\pentsup{(\log L)^3}}$, we use the basic coupling. Then, assuming that the coupling has already been defined above $\T_i$, we do the following within $\T_i$. Outside $\T_{i,1} \cup \T_{i,2} \cup \T_{i,3}$, we use the basic coupling. Since these downward trapezoids do not touch the outer boundary, the dynamics within them do not involve the boundary condition. Moreover, their bases and heights have been chosen in such a way that, for all large $L$, $2a \cdot \mbox{\scriptsize base length } \cdot e^{-b \cdot  \mbox{\scriptsize height}} \leq 1/4$, say.

As a consequence, within $\T_{i,j}$ for $j=1,2,3$, Lemma \ref{l:couplage-dessus} provides a type II coupling such that, for any pair of configurations at the top of $\T_{i,j}$, the probability that they do not lead to the same configuration at the base of $\T_{i,j}$ is bounded above by $2 \cdot 1/4 = 1/2$. Since $\T_{i,j}$, for $j=1,2,3$ do not touch each other except on their boundaries, we may use these couplings independently within $\T_{i,1}, \T_{i,2},\T_{i,3}$. There remain less than $4k$ sites within  $I_1 \cup I_3 \cup I_5 \cup I_7$ that do not belong the the bases of $\T_{i,1}, \T_{i,2},\T_{i,3}$. Invoking the positive rates property of our PCA in conjunction with the basic coupling outside $\T_{i,1} \cup \T_{i,2} \cup \T_{i,3}$, see \eqref{e:positive-rates-coupl}, the probability to have every such site  in a certain state $w \in \A$, for every configuration at the top of $\T_i$, is bounded below by $\kappa^{4k}$, independently of what happens within  $\T_{i,1} \cup \T_{i,2} \cup \T_{i,3}$. As a result, the probability of having the same pair of configurations on the base of $\T_i$ is bounded below by $(1/2)^3 \cdot \kappa^{4k}$.

The coupling is now defined on the whole of $\T$. Starting from a pair of configurations $\zeta_1,\zeta_2$ at the top of $\T$, the probability that all of the $q-\pentsup{(\log L)^3}$ trapezoids $\T_i$ fail to produce the same pair of configurations on their base, is bounded above by $\left( 1-\kappa^{4 k}/8 \right)^{q-\pentsup{(\log L)^3}}$. As soon as  $\eta > 1+\log(1/\kappa) (24/b)$, this quantity is bounded above by $e^{-L/(\log L)^{\eta}}$ for all large enough $L$.

Since there are $|\A|^{M+1}$ distinct initial conditions, using the union bound exactly as in the proof of Lemma \ref{l:couplage-dessus}, the probability of non-coalescence of this coupling is bounded above by $|\A|^{M+1} e^{-L/(\log L)^{\eta}}$. Choosing  any $\theta >\eta$, the inequality $M \leq L/(\log L)^{\theta}$ yields the desired bound on the coalescence probability, and this bound is uniform over $\chi$. 
To get a coupling defined for every boundary condition, we use a version of the coupling just defined for every $\chi$, drawn independently over the various values of $\chi$.

\end{proof}

Now consider the following construction (see Fig. \ref{f:W}): starting from an integer $L$, set $K=2L$, and put side-by-side (from left to right) two downward trapezoids $\T_a, \T_c$ with height $L$ and top length $K$, and put in between  an upward trapezoid $\T_b$ with height $L-1$ and base length $K-2$. Since $K=2L$, these three trapezoids are in fact triangles.

\begin{lemma}\label{l:W}
For arbitrarily large $L$, there exists a $(K,L)$-coupling  within $\T=\T_a \cup \T_b \cup \T_c$ such that: 
\begin{itemize}
\item[\textbullet] The dynamics within $\T_a$ and $\T_c$ are given by two i.i.d. $(K,L)-$couplings.
\item[\textbullet] The dynamics within $\T_b$ is given by an $(M,L-1)$ coupling with $M=0$, independent from the above two $(K,L)-$couplings.
\item[\textbullet] The non-coalescence probability of the overall coupling is bounded above by $e^{-L^{1+o(1)}}$.
\end{itemize}
\end{lemma}

\begin{proof}

Remember the constant $\theta$ from Lemma \ref{l:bc-coupl}, and let $A>\theta$. Now consider an integer $L_1$ (which can be chosen to be arbitrarily large) to which we apply Lemma \ref{l:aplati}, yielding a downward $(K_1,L_1)-$coupling. Note that, by choosing $q$ even in the proof of Lemma \ref{l:self-sim}, we may assume that $K_1$ is an even number, and let $L=L_1+K_1/2$ and $K=2L$. Then let $L_2=L-L_1$, and $M_2=2L_1-2$. For large $L_1$, we have that 
$K_1 \sim h^{-1} L_1 (\log L_1)^A$. We deduce that $L_2 \sim (h^{-1}/2) L_1 (\log L_1)^A$, and $\log L_2 \sim \log L_1$, so that $L_2/(\log L_2)^{\theta} \sim (h^{-1}/2) L_1 (\log L_1)^{A-\theta}$. Since $M_2 = 2 L_1-2$ and $A > \theta$, we see that, for all large enough $L_1$,  $M_2 \leq L_2/(\log L_2)^{\theta}$ so that we may use Lemma \ref{l:bc-coupl} to provide an upward $(M_2,L_2)-$coupling.

Now  (see Fig. \ref{f:W}) put side-by-side two trapezoids downward $\T_1$, $\T_2$ with height $L_1$ and base length $K_1$. Then insert between $\T_1$ and $\T_2$ an upward trapezoid $\T_3$ whose top has the same ordinate as the base of $\T_1$, $\T_2$, with height $L_2$ and base length $K-2$ (so that the top-length is $M_2$). Next, draw two triangles   $\mathscr{T}_1$ and  $\mathscr{T}_2$ just below respectively $\T_1$ and $\T_2$, so that the top of $\mathscr{T}_1$ (resp. $\mathscr{T}_2$) coincides with the base of $\T_1$ (resp. $\T_2$). Finally, let $\T_a=\T_1 \cup \mathscr{T}_1$, $\T_c=\T_2 \cup \mathscr{T}_2$, and let $\T_b$ denote the triangle located between $\T_a$ and $\T_c$, whose boundary is at horizontal distance $1$ from these.

\begin{figure}
\centering
\includegraphics[width=12cm]{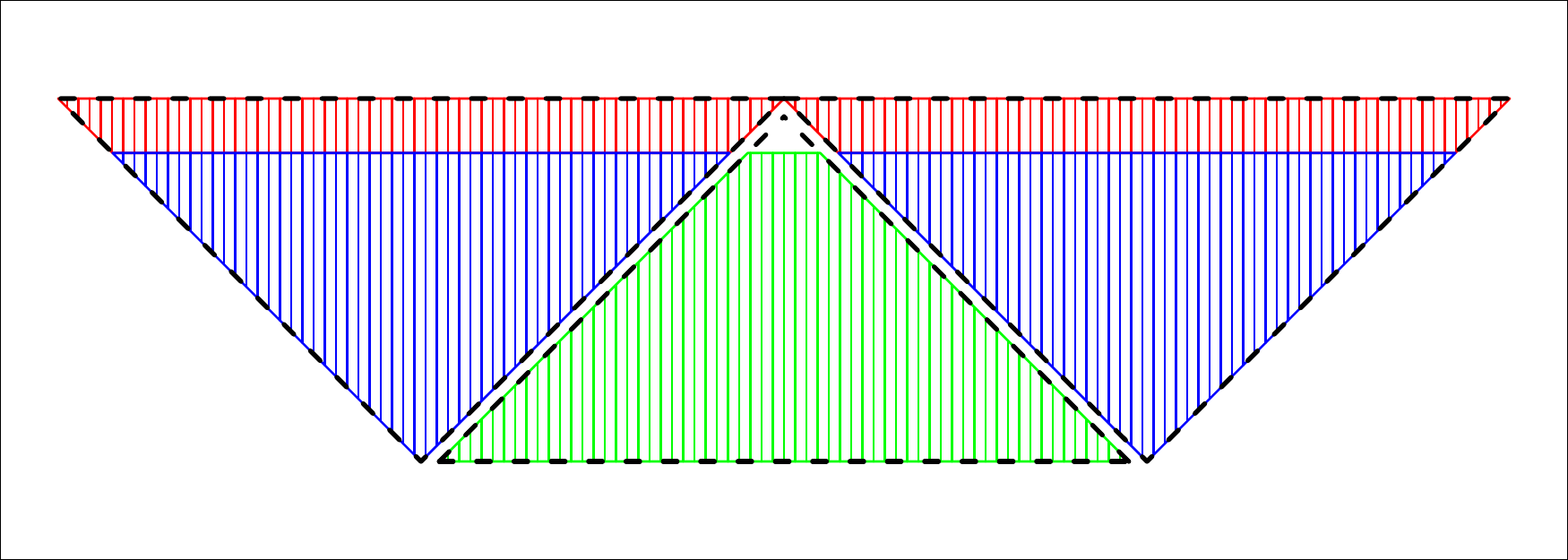}
\caption{Here, the boundaries of $\T_a$ (left), $\T_b$ (middle), $\T_c$ (right) are drawn with dashed lines. Then $\T_1$ (left) and $\T_2$ (right) are drawn in red, while $\T_3$ is drawn in green,  and $\mathscr{T}_1$ (left) and  $\mathscr{T}_2$ (right) are drawn in blue.}
\label{f:W}
\end{figure}
 
 We use the downward $(K_1,L_1)-$coupling provided by Lemma \ref{l:aplati}, independently within $\T_1$ and $\T_2$.
Within  $\mathscr{T}_1$ and  $\mathscr{T}_2$, and also within the triangle located between $\T_1$ and  $\T_2$, we use the basic coupling. Finally, within  $\T_3$, we use the upward $(M_2,L_2)-$coupling from Lemma \ref{l:bc-coupl}, independently from the couplings used within $\T_1$ and $\T_2$. 

The outer boundary of $\T_3$ is included in $\mathscr{T}_1 \cup \mathscr{T}_2$, so the boundary values for the dynamics in $\T_3$ are determined by the couplings we have already defined.

When there is coalescence within $\T_1$ and  $\T_2$, only one (random) boundary condition appears at the outer boundary of $\T_3$, depending solely on the coupling within $\T_a$ and $\T_c$, so that, when in addition there is coalescence within $\T_3$ for this specific boundary condition, there is coalescence within $\T$. Since the coupling within $\T_3$ is independent from the couplings within $\T_a$ and $\T_b$, and since the bound on the non-coalescence probability of the coupling provided by Lemma \ref{l:bc-coupl} is uniform with respect to the boundary condition, the probability not to have coalescence, conditional upon the fact that there is coalescence  within $\T_1$ and  $\T_2$, is bounded above by $e^{-L_2^{1+o(1)}}$. Since the probability not to have coalescence within  $\T_1$ or within $\T_2$ is also bounded above by  $e^{-L_1^{1+o(1)}}$, and since both $\log L_1 \sim \log L$ and $\log L_2 \sim \log L$, we conclude that the overall probability of non-coalescence is bounded above by $e^{-L^{1+o(1)}}$. 
\end{proof}

We are now ready to prove Theorem \ref{t:CFTP}.

\begin{proof}[Proof of Theorem \ref{t:CFTP}]

Remember the definition of $\T$, $\T_a$, $\T_b$ from Lemma \ref{l:W}. We start with a triangle $\T_a$ whose top is $ \llbracket -K,0 \rrbracket  \times \{ L \} $, so that the base of $\T_b$ is $ \llbracket -(L-1),(L-1) \rrbracket \times \{0 \} $, and tile the whole lattice $\Z \times (-\N)$ by translating $\T_a$ and $\T_b$ with vectors of the form $\lambda_1 (2L,0) + \lambda_2 (0,L)$, where $\lambda_1 \in \Z$ and $\lambda_2 \in -\N$. We then define a flow $\Phi_{t_n}^{t_{n+1}}$ for $t_n=-n \cdot L$ by using i.i.d. copies of the couplings provided by Lemma \ref{l:W}, respectively for every copy of $\T_a$, and every copy of $\T_b$. Properties (ii) and (iii) of the theorem are then direct consequences of the definition. 

We now prove (i). Given $x \in \Z$, let $z$ denote  the element of $\Z$ of the form $2kL$ closest to $x$, where $k \in \Z$, and define $I(x)=\llbracket z-2L,z+2L \rrbracket$. For $n \geq 0$, we let $J(x,n)=\emptyset$ if there is coalescence within 
$\T+(z,-Ln)$, and $J(x,n)=I(x)$ otherwise. The sets $J(x,n)$ have been defined in such a way that, for all $m \geq n+1$, 
\begin{equation}\label{e:recure}  \mbox{$\pi_{J(x,n)} \circ \Phi_{t_{n+1}}^{t_m}$ {\scriptsize is a constant function}} \Rightarrow  \mbox{$\pi_{x} \circ \Phi_{t_{n}}^{t_m}$ {\scriptsize is a constant function} },\end{equation}
with the convention that $\Phi_{t_{n+1}}^{t_{n+1}}$ is the identity function.

We then define a random sequence $(P_n)_{n \geq 0}$ of finite subsets of $\Z$, in the following way. We start with $P_0= \{ x \}$.  
Then, assuming $P_0,\ldots, P_n$ have already been defined, we let $P_{n+1} = \bigcup_{x \in P_n} J(x,n)$. 

One checks by induction using \eqref{e:recure} that, if $P_n = \emptyset$, then $\pi_x \circ \Phi_{t_0}^{t_{n}}$ is a constant function, so that we have the bound
$\P(T_x > nL) \leq  \P(P_n \neq \emptyset)$.

Now observe that $| P_{n+1} |\leq \sum_{y \in \Z} \un(y \in P_n) | J(y,n) |$. Moreover, for fixed $y$, $\un(y \in P_n)$ is measurable with respect to $\sigma(\Phi_{t_{k-1}}^{t_{k}}, 1 \leq k \leq n)$, while $| J(y,n) |$ is measurable with respect to $\sigma(\Phi_{t_n}^{t_{n+1}})$. Since the random functions $(\Phi_{t_{k-1}}^{t_{k}})_{k \geq 1}$ form an independent sequence,  $\un(y \in P_n)$ and $| J(y,n) |$ are independent, so we have that  $\E ( | P_{n+1} |) \leq  \sum_{y \in \Z} \P(y \in P_n)  \E(| J(y,n) |)$.
 
In view of the definition of $J(y,n)$, we see that $\E(|J(y,n) |) = (4L+1) \cdot \P(\mbox{non-coalescence in $\T$})=\rho$. As a consequence, we deduce that $\E ( | P_{n+1} |) \leq  \rho  \cdot \sum_{y \in \Z} \P(y \in P_n)  = \rho  \cdot \E(| P_n |)$. Iterating this inequality, we deduce that $\E(| P_n |) \leq \rho^{n}$.

Since the non-coalescence probability is bounded above by $e^{-L^{1+o(1)}}$, $\rho$ can be made arbitrarily small by choosing a large enough value of $L$, and we indeed assume that $\rho$ is $<1$. Using the Markov inequality and the fact that $| P_n |$ is an integer number, we have the following sequence of inequalities, which proves (i):
$$\P(T_x > nL) \leq \P(P_n \neq \emptyset) = \P(| P_n | > 0) =  \P(| P_n | \geq 1) \leq \E(| P_n |) \leq \rho^{n}.$$

\end{proof}

\begin{proof}[Proof of Corollary \ref{c:perturb}]

Assume that $\mathcal{K}$ is a transition kernel defining an exponentially ergodic PCA with positive rates, and let $\mathcal{K}'$ denote a transition kernel distinct from $\mathcal{K}$. Given $\epsilon \in ]0,1[$, assume that $\mathcal{K}'$ is close enough to $\mathcal{K}$ so that, for any $\mathbf{v} \in \A^{\{-1,0,1\}}$ and $w \in \A$, $\mathcal{K}'(\mathbf{v},\{w\}) \geq (1-\epsilon)\mathcal{K}(\mathbf{v},\{w\})$. Letting $\mathcal{K}''(\mathbf{v},\{w\})=\frac{1}{\epsilon}(\mathcal{K}'(\mathbf{v},\{w\})-(1-\epsilon)\mathcal{K}(\mathbf{v},\{w\}))$, we see that $\mathcal{K}''$ is a transition kernel. 

We now reuse the tiling of $\Z \times (-\N)$ with translated copies of $\T_a$ and $\T_b$ used to prove Theorem \ref{t:CFTP}, and define a coupling for the dynamics of $\mathcal{K}'$ as follows. Within each copy of $\T_a$, declare each site in  $\T_a \setminus \mbox{top}(\T_a)$ to be blue with probability $1-\epsilon$ and red with probability $\epsilon$, independently for each site. If all sites are blue, we use within $\T_a$ the coupling defined for the $\mathcal{K}-$dynamics in the proof of Theorem \ref{t:CFTP}. If at least one site is red, we use a version of the basic coupling where red sites use $\mathcal{K}''$ while blue sites use $\mathcal{K}$. A similar construction is done for each copy of $\T_b$. 

We now redo the construction of the sets $(P_n)_{n \geq 0}$ used in the proof of Theorem \ref{t:CFTP} with the following modification: $J(x,n)=\emptyset$ if there is coalescence within 
$\T+(z,Ln)$ and all sites in $\T+(z,Ln)$ are blue, and $J(x,n)=I(x)$ otherwise. As a consequence, $$\E(|J(y,n) |) \leq  (4L+1) \cdot \left((1-\epsilon)^{m} p + (1-(1-\epsilon)^m) \right),$$ where $m=| \T \setminus \mbox{top}(\T)|$
and $p=\P(\mbox{non-coalescence of the $\mathcal{K}-$dynamics in $\T$})$. For large enough $L$, we have that $(4L+1) p <1$. For such an $L$, noting that $m$ depends only on $L$ and not on $\epsilon$, we see that, for all $\epsilon$ small enough so that $\E(|J(y,n) |)<1$, the same argument as in the proof of Theorem  \ref{t:CFTP} leads to the conclusion that the coalescence time $T'_x$ for the $\mathcal{K}'-$dynamics has a finite exponential moment uniformly bounded over $x$. 
\end{proof}

\bibliographystyle{siam}
\bibliography{CFTP-PCA}

\end{document}